\documentclass[oneside,draft,11pt]{amsart}

\usepackage{dsfont}
\usepackage[all]{xy}

\newcommand{\chilow}[1]{\chi_{\lower2pt\hbox{$\scriptstyle#1$}}}
\DeclareMathOperator{\Fin}{Fin}
\DeclareMathOperator{\Ker}{Ker}
\DeclareMathOperator{\Ext}{Ext}
\DeclareMathOperator{\dens}{dens}
\DeclareMathOperator{\w}{w}
\DeclareMathOperator{\clop}{Clop}
\DeclareMathOperator{\MA}{MA}
\DeclareMathOperator{\CH}{CH}
\DeclareMathOperator{\ZFC}{ZFC}
\DeclareMathOperator{\dom}{dom}

\title[Nontrivial twisted sums for finite height spaces]{Nontrivial twisted sums for finite height spaces under Martin's Axiom}

\author{Claudia Correa}
\thanks{The author was partially supported by FAPESP grant 2018/09797-2.}
\address{Centro de Matem\'atica, Computa\c c\~ao e Cogni\c c\~ao,\hfill\break\indent Universidade Federal do ABC, Brazil}
\email{claudiac.mat@gmail.com, claudia.correa@ufabc.edu.br} \urladdr{http://professor.ufabc.edu.br/\~{}claudia.correa}


\date{February 20th, 2019}

\begin{document}

\theoremstyle{plain}\newtheorem{teo}{Theorem}[section]
\theoremstyle{plain}\newtheorem{prop}[teo]{Proposition}
\theoremstyle{plain}\newtheorem{lem}[teo]{Lemma}
\theoremstyle{plain}\newtheorem{cor}[teo]{Corollary}
\theoremstyle{definition}\newtheorem{defin}[teo]{Definition}
\theoremstyle{remark}\newtheorem{rem}[teo]{Remark}

\begin{abstract}
We show that if we assume Martin's Axiom, then there exists a nontrivial twisted sum of $c_0$ and $C(K)$, for every compact space $K$ with finite height and weight at least {\it continuum}. This result settles the problem of existence of nontrivial twisted sums of $c_0$ and $C(K)$, for finite height spaces $K$, under Martin's Axiom.
\end{abstract}

\maketitle

\begin{section}{Introduction}

In this paper we study the existence of nontrivial twisted sums of $c_0$ and $C(K)$, for compact and scattered spaces $K$. As usual, $C(K)$ denotes the space of continuous real-valued functions defined on a compact Hausdorff space $K$, endowed with the supremum norm.
Recall that given Banach spaces $X$ and $Y$ a {\it twisted sum} of $Y$ and $X$ is an exact sequence in the category of Banach spaces of the form
$0 \rightarrow Y \rightarrow Z \rightarrow X \rightarrow 0$,
i.e., $Z$ is a Banach space and the arrows are bounded operators. This twisted sum is said to be {\it trivial} if the image of $Y \rightarrow Z$ is complemented in $Z$. It is clear that there exist nontrivial twisted sums in the category of Banach spaces, since there are uncomplemented subspaces.
An interesting problem is to determine if there are nontrivial twisted sums of $Y$ and $X$, for a fixed pair of Banach spaces $X$ and $Y$. In this context, the space $c_0$ plays a central role due to {\em Sobczyk's Theorem} \cite{Sobczyk} that states that $c_0$ is complemented in every separable superspace.
Therefore, if $X$ is a separable Banach space, then every twisted sum of $c_0$ and $X$ is trivial. A natural question here is about the converse of last implication, i.e., if $X$ is a Banach space such that every twisted sum of $c_0$ and $X$ is trivial, then $X$ is separable? This question is easily answered negatively, since there are nonseparable projective Banach spaces. However it becomes much more interesting in the context of $C(K)$ spaces. Recall that the Banach space $C(K)$ is separable if and only if the compact $K$ is metrizable. Therefore, Sobczyk's Theorem ensures that if $K$ is a metrizable compact space, then every twisted sum of $c_0$ and $C(K)$ is trivial. In this context the converse we are discussing can be rephrased as: Is there a nonmetrizable compact Hausdorff space $K$ such that every twisted sum of $c_0$ and $C(K)$ is trivial? This question was initially proposed in \cite{CastilloKalton} and it remains open. It has been recently studied in a series of papers \cite{Castilloscattered, Lep, ExtCKc0, ValdiviaTree, Plebanek}. In some of those papers, the class of compact spaces studied is the class of scattered spaces. Recall that a topological space $\mathcal X$ is called {\it scattered} if there exists an ordinal $\alpha$ such that its Cantor--Bendixson derivative $\mathcal X^{(\alpha)}$ is empty \cite[Chapter~6]{Koppelberg}. The least ordinal $\alpha$ such that $\mathcal X^{(\alpha)}=\emptyset$ is called the {\it height} of $\mathcal X$ and we say that $\mathcal X$ has {\it finite height} if its height is a natural number. Surprisingly, additional set-theoretic assumptions have played an essential role in those works \cite{Claudia}. The main axioms that appear are the {\it Continuum Hypothesis} ($\CH$) and {\it Martin's Axiom} ($\MA$). In \cite{Plebanek}, G. Plebanek and W. Marciszewski proved that, assuming $\MA+\neg \CH$, there exist nonmetrizable compact Hausdorff spaces $K$ such that every twisted sum of $c_0$ and $C(K)$ is trivial; answering consistently the question we are discussing. More precisely, they showed that if we assume $\MA$, then every twisted sum of $c_0$ and $C(K)$ is trivial, for every compact Hausdorff and separable space $K$ with height 3 and $\w(K)<\mathfrak c$, where $\w(K)$ denotes the weight of $K$ and $\mathfrak c$ denotes the cardinality of the {\it continuum} \cite[Theorem~9.7]{Plebanek}. In \cite[Corollary~4.2]{Lep}, this result was generalized for any finite height. On the other hand, J. Castillo showed that if we assume $\CH$, then there exists a nontrivial twisted sum of $c_0$ and $C(K)$, for every finite height and nonmetrizable compact space $K$ \cite[Theorem~1]{Castilloscattered}. Therefore, the existence of nontrivial twisted sums of $c_0$ and $C(K)$, for finite height, separable and small compact spaces $K$ is independent of the axioms of $\ZFC$. This paper is a continuation of \cite{Lep, Plebanek}. Our main result (Theorem \ref{main}) states that if we assume $\MA+\neg\CH$, then there exists a nontrivial twisted sum of $c_0$ and $C(K)$, for every finite height space $K$ with $\w(K) \ge \mathfrak c$. A fundamental tool in the present work is the study of reminders of compactifications of the discrete space $\omega$, where $\omega$ denotes the set of natural numbers. This study is presented in Section \ref{sec:reminders} and Section \ref{sec:twisted} is devoted to the construction of nontrivial twisted sums.

\end{section}

\begin{section}{Reminders of compactifications of $\omega$}
\label{sec:reminders}

Given a compactification $(\mathcal Y, \varphi)$ of a topological space $\mathcal X$, the {\it reminder} of this compactification is defined as $\mathcal Y \setminus \varphi[\mathcal X]$. Identifying $\varphi[\mathcal X]$ with $\mathcal X$, we denote the reminder by $\mathcal Y \setminus \mathcal X$. An interesting problem in the theory of compactifications and reminders is the problem of determining if a given compact Hausdorff space $K$ is homeomorphic to the reminder of a compactification of the discrete space $\omega$. The first deep result on this problem was obtained by I. Parovi\v{c}enko \cite{Parovicenko} and says that every compact Hausdorff space $K$ with $\w(K) \le \omega_1$ is homeomorphic to the reminder of a compactification of $\omega$, where $\omega_1$ denotes the first uncountable cardinal. It is well known that a compact space $K$ is homeomorphic to the reminder of a compactification of $\omega$ if and only if $K$ is a continuous image of $\beta \omega \setminus \omega$, where $\beta \omega$ denotes the Stone--\v{C}ech compactification of $\omega$. Therefore, if we restrict ourselves to the class of Boolean spaces, then {\it Stone's Duality} \cite[Chapter~3]{Koppelberg} implies that a Boolean space $K$ is homeomorphic to the reminder of a compactification of $\omega$ if and only if the algebra of clopen subsets of $K$ embeds in $\wp(\omega)/\Fin$. As usual $\wp(\omega)/\Fin$ denotes the Boolean algebra obtained by quotienting the Boolean algebra $\wp(\omega)$ by its ideal $\Fin$ formed by the finite subsets of $\omega$. Hence Parovi\v{c}enko's Theorem implies that every Boolean algebra with cardinality at most $\omega_1$ embeds in $\wp(\omega)/\Fin$. Therefore, under $\CH$, we have that every Boolean algebra with cardinality at most $\mathfrak c$ embeds in $\wp(\omega)/\Fin$. This result does not hold in $\ZFC$; for instance, it was shown in \cite{Frank} that it is relatively consistent with $\ZFC$ that there exists a Boolean algebra with cardinality $\mathfrak c$ that does not embed in $\wp(\omega)/\Fin$. Now, let us take a look at this problem under $\MA+\neg \CH$. In \cite{Dowen}, it was shown that under $\MA$ every Boolean algebra with cardinality smaller than $\mathfrak c$ embeds in $\wp(\omega)/\Fin$. However, the same does not hold for Boolean algebras with cardinality $\mathfrak c$.
The Boolean algebra with cardinality $\mathfrak c$ that does not embed in $\wp(\omega)/\Fin$ given in \cite{Frank} is also relatively consistent with $\MA+\neg \CH$. To the best knowledge of the author, it is not known if it is relatively consistent with $\MA+\neg \CH$ that every Boolean algebra with cardinality $\mathfrak c$ embeds in $\wp(\omega)/\Fin$. In this section, we present some contributions to the study of the problem of embedding Boolean algebras with cardinality $\mathfrak c$ in $\wp(\omega)/\Fin$, under $\MA+\neg\CH$.

Let us start by fixing some notations. We denote by $\vert X \vert$ the cardinality of a set $X$. Given $A,B \in \wp(\omega)$ we say that $A$ and $B$ are {\it almost equal} if the symmetric difference $A \Delta B$ is finite and we denote this by $A=^*B$; we say that $A$ and $B$ are {\it almost disjoint} if $A \cap B=^*\emptyset$; we say that $A$ is {\it almost contained} in $B$ if $A \setminus B$ is finite and denote this by $A \subset^*B$. Given a collection $\mathcal A$ of subsets of $\omega$, we denote by $\mathcal I(\mathcal A)$ the ideal of $\wp(\omega)$ formed by the sets that are almost contained in a finite union of elements of $\mathcal A$.

\begin{lem}\label{martin}
Assume $\MA+\neg\CH$. Let $\mathcal A, \mathcal B, \mathcal C \subset \wp(\omega)$ with $\vert \mathcal A\vert \le \omega$, $\vert \mathcal B \vert < \mathfrak c$ and $\vert \mathcal C \vert < \mathfrak c$. Assume that:
\begin{enumerate}
\item[(i)] $A \cap B=^*\emptyset$, for every $A \in \mathcal A$ and every $B \in \mathcal B$;
\item[(ii)] $\mathcal C \cap \mathcal I(\mathcal A)=\emptyset$;
\item[(iii)] $\mathcal C \cap \mathcal I(\mathcal B)=\emptyset$.
\end{enumerate}
Then there exists $S \in \wp(\omega)$ such that $A \subset^* S$, for every $A \in \mathcal A$, $B \cap S=^*\emptyset$, for every $B \in \mathcal B$, and $C \cap S$ and $C \cap (\omega \setminus S)$ are infinite, for every $C \in \mathcal C$.
\end{lem}
\begin{proof}
Let $\mathbb P$ denote the set of functions $f:\dom(f)\subset \omega \longrightarrow 2$ satisfying $f^{-1}(1)=^*\bigcup_{i=1}^n A_i$ and $f^{-1}(0)=^*\bigcup_{j=1}^m B_j$, where $n, m \in \omega$, $A_i \in \mathcal A$ and $B_j \in \mathcal B$. Endow $\mathbb P$ with the partial order defined as $f \leq g$ iff $g \subset f$. Note that if $f,g \in \mathbb P$ satisfy $f^{-1}(1)=g^{-1}(1)$, then they are compatible, therefore the fact that $\mathcal A$ is countable implies that $(\mathbb P, \leq)$ has ccc. Consider the following dense subsets of $\mathbb P$:
\[D_A=\{f \in \mathbb P: A \subset^* f^{-1}(1)\}, \ D_B=\{f \in \mathbb P: B \subset^* f^{-1}(0)\} \ \text{and}\]
\[D_{C,n}=\{f \in \mathbb P: \exists i,j>n \ \text{with} \ i \in C \cap f^{-1}(1) \ \text{and} \ j \in C \cap f^{-1}(0)\},\]
for every $A \in \mathcal A$, $B \in \mathcal B$, $C \in \mathcal C$ and $n \in \omega$. Let $G$ be a filter in $\mathbb P$ that intersects all those dense sets whose existence is ensured by $\MA$ and define the function $\varphi=\bigcup G$. The desired set is given by $S=\varphi^{-1}(1)$.
\end{proof}

Let $\mathfrak B$ be a Boolean algebra. We denote by $\langle X \rangle$ the subalgebra of $\mathfrak B$ generated by its subset $X$ and by $\mathcal I(b)$ the {\it ideal generated} by $b \in \mathfrak B$, i.e., $\mathcal I(b)=\{c \in \mathfrak B: c \le b\}$.

\begin{teo}\label{resultadoBA}
Assume $\MA+\neg \CH$. Let $\mathfrak B$ be a Boolean algebra with $\vert \mathfrak B \vert=\mathfrak c$ and let $X$ be a set of generators of $\mathfrak B$.
If $\vert \{b \in X: \vert \mathcal I(b)\vert> \omega\} \vert <\mathfrak c$, then $\mathfrak B$ embeds in $\wp(\omega)/\Fin$.
\end{teo}
\begin{proof}
If $\mathfrak{B}_0=\big \langle \{b \in X: \vert \mathcal I(b)\vert>\omega\} \big \rangle$, then $\MA$ implies that there exists an embedding $\varphi:\mathfrak{B}_0 \rightarrow \wp(\omega)/\Fin$ \cite{Dowen}. Enumerate $\{b \in X: \vert \mathcal I(b)\vert \le \omega\}$ as $\{b_\alpha: \alpha \in \mathfrak c\}$. By recursion on $\alpha \in \mathfrak c$, we shall construct an increasing family $(\varphi_\alpha)_{\alpha \in \mathfrak c}$, where each $\varphi_\alpha$ is an embedding from $\big \langle \mathfrak{B}_0 \cup \{b_\beta: \beta \le \alpha\}\big \rangle$ in $\wp(\omega)/\Fin$ extending $\varphi$. Then the embedding of $\mathfrak B$ in $\wp(\omega)/\Fin$ will be given by $\bigcup_{\alpha \in \mathfrak c} \varphi_\alpha$. For every $\alpha \in \mathfrak c$, set $\mathfrak B_{\alpha}=\big \langle \mathfrak{B}_0 \cup \{b_\beta: \beta < \alpha\}\big \rangle$. Assume that we have already defined $\varphi_\beta$, for every $\beta \in \alpha$, then $\psi=\bigcup_{\beta \in \alpha}\varphi_\beta$ is an embedding from $\mathfrak{B}_\alpha$ in $\wp(\omega)/\Fin$ that extends $\varphi$. Recall that {\em Sikorski's extension Theorem} \cite{Koppelberg} ensures that there is an embedding of $\mathfrak{B}_{\alpha+1}$ into $\wp(\omega)/\Fin$ extending $\psi$ if and only if there is $S \in\wp(\omega)$ satisfying:
\begin{enumerate}
\item[(a)] $\psi(b) \le q(S)$ if and only if $b \le b_\alpha$, for every $b \in \mathfrak{B}_\alpha$;
\item[(b)] $\psi(b) \wedge q(S)=0$ if and only if $b \wedge b_\alpha=0$, for every $b \in \mathfrak{B}_\alpha$,
\end{enumerate}
where $q:\wp(\omega)\rightarrow\wp(\omega)/\Fin$ denotes the quotient map. Set:
\[\mathcal A=q^{-1}\big[\{ \psi(b): b \in \mathfrak{B}_\alpha \ \text{and} \ b \le b_\alpha\}\big],\]
\[\mathcal B=q^{-1}\big[\{ \psi(b): b \in \mathfrak{B}_\alpha \ \text{and} \ b \wedge b_\alpha=0\}\big] \ \text{and}\]
\[\mathcal C=q^{-1}\big[\{ \psi(b): b \in \mathfrak{B}_\alpha, b \wedge b_\alpha \ne 0 \ \text{and} \ b \wedge b'_\alpha \ne 0\}\big].\]
It is clear that $\vert \mathcal B \vert<\mathfrak c$, $\vert \mathcal C \vert<\mathfrak c$ and $\vert \mathcal A \vert \le \omega$. Conditions (i), (ii) and (iii) of Lemma \ref{martin} follow from the fact that $\psi$ is an embedding. Let $S$ be the set given by Lemma \ref{martin} and note that $S$ satisfies conditions (a) and (b) above.
\end{proof}

In order to apply Theorem \ref{resultadoBA} to Boolean spaces, we present a simple characterization of sets of generators of the clopen algebra $\clop(K)$, for any Boolean space $K$. Recall that a collection $\mathfrak F$ of functions defined on a set $X$ is said to {\it separate the points of $X$} if given $x_1,x_2 \in X$ with $x_1 \ne x_2$, there exists $f \in \mathfrak F$ such that $f(x_1) \ne f(x_2)$.

\begin{lem}\label{separapontos}
Let $K$ be a Boolean space and $\{C_i: i \in I\} \subset \clop(K)$. Then $\{C_i: i \in I\}$ generates $\clop(K)$ if and only if $\{\chilow{C_i}: i \in I\}$ separates the points of $K$, where $\chilow{C_i}$ denotes the characteristic function of $C_i$.
\end{lem}
\begin{proof}
Let $\varphi:K \rightarrow 2^I$ be the continuous map whose $i$-th coordinate is $\chilow{C_i}$ and denote by $\varphi^*:\clop(2^I) \to \clop(K)$ the corresponding map under Stone's Duality. The conclusion is obtained by noting that the image of $\varphi^*$ is the subalgebra generated by $\{C_i: i \in I\}$
and that $\varphi$ is injective if and only if $\{\chilow{C_i}: i \in I\}$ separates the points of $K$.
\end{proof}

Given a topological space $\mathcal X$ and $x \in \mathcal X$, we define:
\[\w(\mathcal X,x)=\min \{\w(V): \ \text{$V$ is a nhood of $x$}\}.\]

\begin{prop}\label{zeroDim}
Assume $\MA+\neg \CH$. Let $K$ be a Boolean space such that $\w(K)=\mathfrak c$. If $\w\big(\{p \in K: \w(K,p)>\omega\}\big) < \mathfrak c$, then $K$ is homeomorphic to the reminder of a compactification of $\omega$.
\end{prop}
\begin{proof}
Define $F=\{p \in K: \w(K,p)>\omega\}$ and note that $F$ is a closed subset of $K$. For each $C \in \clop(F)$, choose an element $\tilde C$ in $\clop(K)$ such that $C= \tilde C \cap F$ and set $X=\{\tilde C: C \in \clop(F)\} \cup \{C \in \clop(K): C \subset K \setminus F\}$.
It is clear that $\mathcal I(C)=\clop(C)$ is countable, for every $C \in \clop(K)$ with $C \subset K \setminus F$. The result follows from Theorem \ref{resultadoBA} and Lemma \ref{separapontos}, since $\vert \clop(F) \vert=\w(F)<\mathfrak c$ and $\{\chilow{C}: C \in X\}$ separates the points of $K$.
\end{proof}

\begin{rem}\label{fatosDispersos}
Recall that if a Boolean space is scattered, then its weight coincides with its cardinality and any of its closed subspaces is again a scattered Boolean space \cite{Koppelberg}.
\end{rem}

\begin{cor}\label{dispersoReminder}
Assume $\MA+\neg\CH$. Let $K$ be a compact Hausdorff and scattered space with $\w(K)=\mathfrak c$. If $\vert \{p \in K: \w(K,p)> \omega\} \vert < \mathfrak c$, then $K$ is homeomorphic to the reminder of a compactification of $\omega$.
\end{cor}
\begin{proof}
It follows from Remark \ref{fatosDispersos} and Proposition \ref{zeroDim}.
\end{proof}

\begin{rem}
Note that it does not hold in $\ZFC$ that every scattered compact space with weight at most $\mathfrak c$ is homeomorphic to the reminder of a compactification of $\omega$. Indeed, K. Kunen \cite{Kunen} showed that it is relatively consistent with $ZFC$ that $\omega_2 \le \mathfrak c$ and that the ordinal segment $[0,\omega_2]$ is not homeomorphic to the reminder of a compactification of $\omega$ . However, Kunen's result is not relatively consistent with $\MA+\neg\CH$. It seems to be unknown if, under $\MA+\neg CH$, every scattered compact space with weight $\mathfrak c$ is homeomorphic to the reminder of a compactification of $\omega$. Therefore, Corollary \ref{dispersoReminder} is an interesting contribution to this open problem.
\end{rem}

\end{section}

\begin{section}{Nontrivial twisted sums}
\label{sec:twisted}

Let us now explain the relationship between reminders of compactifications of $\omega$ and twisted sums of Banach spaces. If $\gamma \omega$ is a compactification of $\omega$, then we have a natural isometric copy of $c_0$ in $C(\gamma \omega)$; this copy consists of the elements of $C(\gamma \omega)$ that vanish in the reminder $\gamma \omega \setminus \omega$. Following \cite{Plebanek}, we say that a compactification $\gamma \omega$ of $\omega$ is {\it tame} if this natural copy of $c_0$ is complemented in $C(\gamma \omega)$. Therefore, if the compactification $\gamma \omega$ is not tame, then there exists a nontrivial twisted sum of $c_0$ and $C(\gamma \omega \setminus \omega)$, since $C(\gamma \omega)$ quotiented by the natural copy of $c_0$ is isometric to $C(\gamma \omega \setminus \omega)$. In \cite[Theorem~2.8]{Plebanek}, the following necessary condition for the tameness of a compactification is presented.

\begin{lem}\label{Kubis}
If a compactification $\gamma \omega$ of $\omega$ is tame, then its reminder $\gamma \omega \setminus \omega$ carries a strictly positive measure.
\end{lem}

It is well known that if $K$ is a nonseparable scattered compact space, then $K$ does not carry a strictly positive measure. Using this and Parovi\v{c}enko's Theorem, it was shown in \cite[Corrolary~9.6]{Plebanek} that under $\CH$ there exists a nontrivial twisted sum of $c_0$ and $C(K)$, for every nonseparable compact scattered space $K$ with $\w(K)=\mathfrak c$. In the next proposition, we present a version of this result under $\MA+\neg \CH$.

\begin{prop}\label{MApesopequenoNaoSep}
Assume $\MA+\neg \CH$. Let $K$ be a compact Hausdorff and scattered space with $\w(K)<\mathfrak c$. If $K$ is nonseparable, then there exists a nontrivial twisted sum of $c_0$ and $C(K)$.
\end{prop}
\begin{proof}
It follows from the discussion above, Lemma \ref{Kubis} and the fact that $\MA+\neg\CH$ implies that every compact space $K$ with $\w(K)<\mathfrak c$ is homeomorphic to the reminder of a compactification of $\omega$ \cite{Dowen}.
\end{proof}

Another class of compact spaces discussed in \cite{Plebanek} that does not carry a strictly positive measure is the class of nonseparable compact lines. Recall that a {\it compact line} is a totally ordered space that is compact, when endowed with the order topology. It was shown in \cite[Theorem~8.1]{Plebanek} that if we assume $\CH$, then there exists a nontrivial twisted sum of $c_0$ and $C(K)$, for every nonseparable compact line $K$. In Proposition \ref{retaCompacta}, we prove that this actually holds in $\ZFC$.

\begin{prop}\label{retaCompacta}
If $K$ is a nonseparable compact line, then there exists a nontrivial twisted sum of $c_0$ and $C(K)$.
\end{prop}
\begin{proof}
The case when $K$ does not have ccc was already solved in \cite[Theorem~8.1]{Plebanek}. If $K$ has ccc, then $K$ is first-countable \cite[3.12.4]{Engelkin}. In \cite[Corollary~3.2]{Bell}, M. Bell proved that every first-countable compact line is homeomorphic to the reminder of a compactification of $\omega$, therefore the result follows from the fact that $K$ does not carry a strictly positive measure \cite{Sap} and from Lemma \ref{Kubis}.
\end{proof}

A fundamental ingredient to establish our main result (Theorem \ref{main}) is Lemma \ref{reduzPeso} below. This result is an adaptation of \cite[Theorem~9.1]{Plebanek}. It is interesting to observe that even though the proof of \cite[Theorem~9.1]{Plebanek} is purely topological, we obtained Lemma \ref{reduzPeso} using homological tools, inspired by \cite{Castilloscattered}. For the definitions of the homological objects that appear in the proof of Lemma \ref{reduzPeso}, see \cite{Castilloscattered, ThreeSpace}. Given topological spaces $\mathcal X$ and $\mathcal Y$, we denote by $C(\mathcal X,\mathcal Y)$ the set of continuous functions from $\mathcal X$ to $\mathcal Y$. As usual, $\dens(\mathcal X)$ denotes the density of $\mathcal X$. Moreover, given Banach spaces $X$ and $Y$, we denote by $\mathcal{L}(X,Y)$ the space of bounded operators from $X$ to $Y$.

\begin{lem}\label{reduzPeso}
Let $K$ be a compact Hausdorff and scattered space. Assume that there exists a closed subset $F$ of $K$ satisfying:
\begin{enumerate}
\item[(a)] $F$ is homeomorphic to the reminder of a compactification of $\omega$;
\item[(b)] $\vert C(F,F) \vert>\mathfrak c$.
\end{enumerate}
Then there exists a nontrivial twisted sum of $c_0$ and $C(K)$.
\end{lem}
\begin{proof}
Note that condition (a) implies that $\w(F) \le \mathfrak c$, hence Remark \ref{fatosDispersos} ensures that $\vert F \vert \le \mathfrak c$. Initially, we will show that $\big \vert \Ext\big(C(F),c_0\big) \vert> \mathfrak c$. It holds that $\vert \Ext\big(C(F),c_0\big) \vert=\vert\mathcal{L} \big(C(F),\ell_\infty/c_0\big)/W\vert$, where $W$ denotes the closed subspace of $\mathcal{L} \big(C(F),\ell_\infty/c_0\big)$ consisting of the operators that admit a lifting to $\ell_\infty$ \cite[Proposition~1.4.f]{ThreeSpace}. To establish our estimate, we will prove that $\vert \mathcal{L} \big(C(F),\ell_\infty/c_0\big)\vert>\mathfrak c$ and that $\vert W \vert \le \mathfrak c$. Clearly $\vert W \vert \le \big\vert \mathcal{L}\big(C(F),\ell_\infty\big)\big \vert$ and it is easy to see that $\vert \mathcal{L}\big(C(F),\ell_\infty\big)\big \vert \le\vert C(F)^* \vert ^\omega$. Therefore we conclude that $\vert W \vert \le\mathfrak c$, since $\vert C(F)^* \vert \le \mathfrak c$ \cite[Theorem~12.28]{Fabian}.
Let $\varphi:\beta \omega \setminus \omega \rightarrow F$ be the continuous onto map given by (a) and note that the following maps are injective:
\[C(F,F) \overset{T}\longrightarrow C(\beta \omega \setminus \omega,F) \overset{S}\longrightarrow \mathcal{L}\big(C(F),C(\beta \omega \setminus \omega)\big),\]
where $T(f)=f \circ \varphi$, $S(\phi)=\phi^*$ and $\phi^*:C(F) \rightarrow C(\beta \omega \setminus \omega)$ is given by $\phi^*(g)=g \circ \phi$. Then condition (b) implies that $\vert \mathcal{L} \big(C(F),\ell_\infty/c_0\big)\vert>\mathfrak c$. Now consider the restriction operator $R:C(K) \rightarrow C(F)$. It follows from \cite[Lemma~4]{Castilloscattered} that there exist a closed subspace $X$ of $C(K)$ and a closed subspace $Y$ of $\Ker(R)$ such that $\dens(Y) \le \dens\big(C(F)\big)$ and such that the diagram below has exact rows and commutes:
\begin{equation}\label{diagram}
\xymatrix@C+8pt@R+4pt{
0\ar[r] &\Ker(R) \vbox to 10pt{\vfil}\ar@{^{(}->}[r] & C(K) \ar[r]^{R} & C(F) \ar[r] &0\\
0 \ar[r] &Y \vbox to 10pt{\vfil}\ar@{^{(}->}[u]\vbox to 5pt{\vfil}\ar@{^{(}->}[r] &X \vbox to 10pt{\vfil}\ar@{^{(}->}[u] \ar[r]^{R|_X} &C(F) \ar[u]_{\text{identity}} \ar[r] &0}
\end{equation}
Assuming by contradiction that $\Ext \big(C(K),c_0\big)$ is zero, one can argue as in the proof of \cite[Lemma~5]{Castilloscattered}
using diagram \eqref{diagram} to obtain a surjective map from $\mathcal{L}(Y, c_0)$ to $\Ext\big(C(F),c_0\big)$. However this is a contradiction, because:
\[\vert \mathcal{L}(Y, c_0) \vert \le \vert Y^* \vert^\omega \le \mathfrak c,\]
where the last inequality follows from the fact that $\dens(Y^*)=\dens(Y)$, since $C(K)$ is Asplund (see \cite[Theorem~12.29]{Fabian} and \cite[Theorem~6]{Yost}).
\end{proof}

We are now ready to prove our main result.
\begin{teo}\label{main}
Assume $\MA+\neg\CH$. If $K$ is a compact Hausdorff and scattered space with finite height and $\w(K) \ge \mathfrak c$, then there exists a nontrivial twisted sum of $c_0$ and $C(K)$.
\end{teo}
\begin{proof}
Denote by $N$ the height of $K$. Firstly, let us prove that we can assume that $K^{(N-1)}$ has only one point. Write $K^{(N-1)}=\{p_1,\ldots,p_k\}$, where $k=\vert K^{(N-1)}\vert$. Clearly, there exist disjoint clopen subsets $C_1,\ldots,C_k$ of $K$ such that $K=\bigcup_{i=1}^{k} C_i$ and $p_i \in C_i$, for every $i=1,\ldots,k$. It is easy to see that $C_i^{(N-1)}=\{p_i\}$, for every $i=1,\ldots,k$ and that there exists an index $i_0$ such that $w(C_{i_0}) \ge \mathfrak c$. Since the canonical copy of $C(C_{i_0})$ in $C(K)$ is complemented, a nontrivial twisted sum of $c_0$ and $C(C_{i_0})$ yields a nontrivial twisted sum of $c_0$ and $C(K)$. Set $I_j=K^{(j)} \setminus K^{(j+1)}$, for $j=0,\ldots,N-2$. For each $j$ and each $\gamma \in I_j$, since $\gamma$ is isolated in $K^{(j)}$, we can choose $S^j_\gamma \in \clop(K)$ with $S^j_\gamma \cap K^{(j)}=\{\gamma\}$ and $\w(S_\gamma^j)=\vert S_\gamma^j\vert=\w(K, \gamma)$. We proceed by induction on the height of $K$. Since $w(K) \ge \mathfrak c$, the first step of the inductive process is $N=2$. In this case $K$ is the one-point compactification of a discrete space $I$ of size greater or equal to $\mathfrak c$. It is well known that $C(K)$ is isomorphic to $c_0(I)$ and that there exists a nontrivial twisted sum of $c_0$ and $c_0(I)$ \cite[\S 2]{Castilloscattered}. Now fix $N \ge 3$ and note that if there exist $j\le N-2$ and $\gamma \in I_j$ such that $\vert S_\gamma^j \vert \ge \mathfrak c$, then the induction hypothesis implies that there exists a nontrivial twisted sum of $c_0$ and $C(S_\gamma^j)$. Since $C(S_\gamma^j)$ is complemented in $C(K)$, it follows that there exists a nontrivial twisted sum of $c_0$ and $C(K)$. Now assume that $\vert S_\gamma^j \vert<\mathfrak c$, for every $j\le N-2$ and every $\gamma \in I_j$. Let $j$ be the greatest natural number with $\vert I_j \vert \ge \mathfrak c$.  and fix a subset $I$ of $I_j$ with $\vert I \vert=\mathfrak c$. Consider the closed subset $F=K^{(j+1)} \cup I$ of $K$ and note that
\[\vert \{p \in F: \w(F,p)>\omega\} \vert < \mathfrak c.\]
Therefore, Corollary \ref{dispersoReminder} ensures that $F$ is
homeomorphic to the reminder of a compactification of $\omega$. Finally, to conclude the result using Lemma \ref{reduzPeso}, let us prove that $\vert C(F,F) \vert > \mathfrak c$. Set $A= I \setminus \bigcup_{k=j+1}^{N-2} \bigcup_{\gamma \in I_k}S_\gamma^k$ and note that $\vert A \vert=\mathfrak c$, since $\mathfrak c$ is regular under $\MA$. Moreover any bijection $h:A \rightarrow A$ induces a continuous map $\varphi_h: F\rightarrow F$ that fixes every point of $F \setminus A$ and such that $\varphi_h|_A=h$.
\end{proof}

\begin{rem}
Note that the problem of existence of nontrivial twisted sums of $c_0$ and $C(K)$, for finite height spaces $K$ is now settled under $\MA$. Indeed, in \cite{Castilloscattered}, it was solved under $\CH$. Under $\MA+\neg \CH$, if $\w(K)<\mathfrak c$, then the separable case was solved in \cite{Lep, Plebanek} and the nonseparable case is solved by Proposition \ref{MApesopequenoNaoSep}. Finally, the case when $\w(K) \ge \mathfrak c$ is solved by Theorem \ref{main}.
\end{rem}

\end{section}

\noindent\textbf{Acknowledgments.}\enspace The author wishes to thank Daniel Victor Tausk for valuable discussions during the preparation of this work.


\begin{thebibliography}{99}

\bibitem{Bell} M. G. Bell, {\em A first countable compact space that is not an $N^*$ image}, Topology Appl. 35 (1990), 153--156.

\bibitem{CastilloKalton} F. Cabello, J. M. F. Castillo, N. J. Kalton, and D. T. Yost, {\em Twisted sums with $C(K)$ spaces}, Trans.\ Amer.\ Math.\ Soc.
355 (2003), 4523--4541.

\bibitem{Castilloscattered} J. M. F. Castillo, {\em Nonseparable $C(K)$-spaces can be twisted when $K$ is a finite height compact}, Topology Appl. 198 (2016), 107--116.

\bibitem{ThreeSpace} J. M. F. Castillo and M. González, {\em Three-space problems in Banach space theory}, Springer, Berlin, 1997.

\bibitem{Claudia} C. Correa, {\em Additional set-theoretic assumptions and twisted sums of Banach spaces}, in: Logic around the world: On the occasion of 5th Annual Conference of the Iranian Association of Logic, M. Pourmahdian and A. S. Daghighi (eds.), A.F.J. Publishing, Tehran, 2017, 73--86.

\bibitem{Lep} C. Correa and D. V. Tausk, {\em Local extension property for finite height spaces}, to appear in Fund.\ Math. The preprint can be found at
arxiv.org/pdf/1801.08619.pdf

\bibitem{ExtCKc0} C. Correa and D. V. Tausk, {\em Nontrivial twisted sums of $c_0$ and $C(K)$}, J. Funct.\ Anal. 270 (2016), 842--853.

\bibitem{ValdiviaTree} C. Correa and D.V. Tausk, {\em Small Valdivia compacta and trees}, Studia Math. 235 (2016), 117--135.

\bibitem{Dowen} E. van Douwen and T. Przymusi\' nski, {\em Separable extensions of first countable spaces}, Fund.\ Math. 95 (1980), 147--158.

\bibitem{Engelkin} R. Engelking, {\em General Topology}, Heldermann Verlag, Berlin, 1989.

\bibitem{Fabian} M. Fabian, P. Habala, P. H\'ajek, V. Montesinos and V. Zizler, {\em Functional Analysis and Infinite-Dimensional Geometry}, Springer-Verlag, New York, 2001.

\bibitem{Frank} R. Frankiewicz, {\em Some remarks on embeddings of Boolean algebras and topological spaces II}, Fund.\ Math. 126 (1985), no.~1, 63--68.

\bibitem{Koppelberg} S. Koppelberg, {\em Handbook of Boolean algebras}, Vol. 1, R. Bonnet and J. D. Monk (eds.), North-Holland, Amsterdam, 1989.

\bibitem{Kunen} K. Kunen, {\em Inaccessibility properties of cardinals}, Doctoral Dissertation, Stanford, 1968.

\bibitem{Plebanek} W. Marciszewski and G. Plebanek, {\em Extension operators and twisted sums of $c_0$ and $C(K)$ spaces}, J.\ Funct.\ Anal. 274 (2018), 1491--1529.

\bibitem{Parovicenko} I. Parovi\v{c}enko, {\em A universal bicompact of weight $\aleph_1$}, Soviet Mathematics Doklady 4 (1963), 592--592, Russian original: Ob odnom universal'nom bikompakte vesa $\aleph$, Doklady Akademii Nauk SSSR 150 (1963) 36--39.

\bibitem{Sap} A. Sapounakis, {\em Measures on totally ordered spaces}, Mathematika 27 (1980), 225--235.

\bibitem{Sobczyk} A. Sobczyk, {\em Projections of the space $(m)$ on its subspace $(c_0)$}, Bull.\ Amer.\ Math.\ Soc. 47 (1941), 938--947.

\bibitem{Yost} D. Yost, {\em Asplund spaces for beginners}, Acta Univ. Carolin. Math. Phys. 34 (1993), no.~2, 159--177.

\end{thebibliography}
\end{document}